\documentclass[12pt,onesided]{article}

\usepackage[leqno]{amsmath}
\usepackage{amsthm}
\usepackage{amssymb}
\usepackage{amsfonts}
\usepackage{array}
\usepackage{hyperref}

\newcommand{\f}{\mathcal{F}}
\newcommand{\p}{\mathbb{P}}

\newtheorem{theorem}{THEOREM}

\newtheorem{definition}{DEFINITION}

\def\bi{\begin{itemize}}
\def\ei{\end{itemize}}

\begin{document}

\title{Stopping times are hitting times:\\ a natural representation}
\author{Tom Fischer\thanks{Institute of Mathematics, University of Wuerzburg, 
Campus Hubland Nord, Emil-Fischer-Strasse 30, 97074 Wuerzburg, Germany.
Tel.: +49 931 3188911.
E-mail: {\tt tom.fischer@uni-wuerzburg.de}.
}\\
{University of Wuerzburg}} 
\date{This version: \today\\First version: December 7, 2011}
\maketitle

\begin{abstract}
There exists a simple, didactically useful one-to-one relationship between stopping times 
and adapted c\`adl\`ag (RCLL) processes that are non-increasing and take the values 0 and 1 only. 
As a consequence, stopping times are always hitting times. 
\end{abstract}




The possibly most popular example of a stopping time is the `hitting time' of a set by a
stochastic process, defined as the first time at which a certain pre-specified set is hit 
by the considered process. Often this example is
considered for a Borel-measurable set and a one-dimensional real-valued adapted process on 
a discrete time axis.
Similar results in more general set-ups are usually collectively called the {\em D\'ebut} Theorem;
see Bass (2010, 2011) for a proof.

Astonishingly, it seems to be less widely taught (and maybe known) that the inverse is true as well: 
for any stopping time there exists an adapted stochastic process and
a Borel measurable set such that the corresponding hitting time will be exactly this stopping time.
Furthermore, the stochastic process can be chosen very intuitively: 
it will be 1 until just before the stopping time is reached, from which on it will be 0.
The 1-0-process therefore first hits the Borel set $\{0\}$ at the stopping time. 
As such, a one-to-one relationship is established between stopping times and adapted 
c\`adl\`ag processes that are non-increasing and take the values 0 and 1 only. 

A 1-0-process as described above is obviously akin to a random traffic light which can go through 
three possible scenarios over time (here, `green' stands for `$1$' and `red' stands for `$0$'):
(i) it stays red forever (`stopped immediately');
(ii) it is green at the beginning, then turns red, and stays red for the rest of the time (`stopped at some stage');
(iii) it stays green forever (`never stopped').
So, the traffic light can never change back to green once it has turned red 
(stopped once means stopped forever), and, for the adaptedness of
the 1-0-process, it can only change based on information up to the corresponding point in time.
This very intuitive interpretation of a stopping time as the time when such a `traffic light'
changes is considerably easier to understand than the concept of a random time which is 
`known once it has been reached' -- one of the verbal interpretations of the usual standard definition 
of a stopping time (see Def.~\ref{def_1} below).

While this representation and alternative definition of stopping times seems natural and didactically useful, it does not 
seem to be widely taught, or otherwise one would expect to find it in textbooks. 
However, there is no mentioning of it in standard textbooks on probability such as 
Billingsley (1995) or Bauer (2001) 
(Bauer as an example of a popular German stochastics textbook) or in
standard textbooks on stochastic processes/calculus and stochastic mathematical finance 
such as Karatzas and Shreve (1991) or Bingham and Kiesel (2004).\\

We denote the time axis by $\mathbb{T}$, where $\mathbb{T}\subset\mathbb{R}$.

\begin{definition}
\label{def_1}
A stopping time w.r.t.~to a filtration $\mathbb{F}=(\f_t)_{t\in\mathbb{T}}$ 
on a probability space $(\Omega, \f_\infty, \p)$ is a random variable
with values in $\mathbb{T}\cup\{+\infty\}$ such that
\begin{equation}
\label{st8}
\{\tau \leq t\} \in \f_t  \quad (t\in\mathbb{T}) .
\end{equation}
\end{definition}

\eqref{st8} means that at any time $t\in\mathbb{T}$ one knows -- based on information up to time $t$ -- 
if one has been stopped already, or not. Note that $\f_t \subset \f_{\infty}$ is assumed for $t\in\mathbb{T}$.

In the following, we call a real-valued stochastic process c\`adl\`ag 
(French:~{\em continue \`a droite, limit\'ee \`a gauche}) if for {\em all} $\omega\in\Omega$ the paths 
$X_{\cdot}(\omega): \mathbb{T} \rightarrow \mathbb{R}$
have the property of being right-continuous with left-handed limits (RCLL).

\begin{definition}
We call an adapted c\`adl\`ag process 
$X = (X_t)_{t\in\mathbb{T}}$ on $(\Omega, \f_\infty, \mathbb{F}, \p)$
a `stopping process' if
\begin{equation}
\label{st3}
X_t(\omega) \in \{0, 1\} \quad (\omega\in\Omega,\, t\in\mathbb{T})
\end{equation} and
\begin{equation}
\label{st4}
\quad X_s \geq X_t \quad (s \leq t;\, s,t\in\mathbb{T}) .
\end{equation}
\end{definition}

Obviously, $\{X_t = 1\}\in\f_t$ for $t\in\mathbb{T}$. For a finite or infinite discrete time axis given by 
$\mathbb{T} = \{t_k:\; k\in\mathbb{N},\, t_k \geq t_j \text{ if } k \geq j\}$, an
adapted process fulfilling \eqref{st3} and \eqref{st4} is automatically c\`adl\`ag. 
This follows from the observation
that in this case $\mathbb{T}$ is bounded from below and either (1) finite as a set, or 
(2) is infinite and has one accumulation point, which lies not in $\mathbb{T}$, and therefore is bounded, or (3)
has no accumulation point and is unbounded from above.

\begin{definition}
For a stopping process $X$ on $(\Omega, \f_\infty, \mathbb{F}, \p)$, define
\begin{equation}
\label{st6}
\tau^X (\omega) = 
\begin{cases}
+\infty & \quad \text{if } X_t(\omega) = 1 \; \text{ for all } t\in\mathbb{T}, \\
\min \{t\in\mathbb{T}: X_t(\omega) = 0\} & \quad \text{otherwise} .
\end{cases}
\end{equation}
\end{definition}

The minimum in the lower case exists because of the c\`adl\`ag property for each path.
By definition, it is clear that
\begin{equation}
\label{st2}
X_t = \mathbf{1}_{\{\tau^X > t\}} \quad (t\in\mathbb{T}),
\end{equation}
which by adaptedness of $X$ and $\{\tau^X > t\} = \{\tau^X \leq t\}^C$ implies 
\begin{equation}
\label{st7}
\{\tau^X \leq t\} \in \f_t  \quad (t\in\mathbb{T}) .
\end{equation}
Therefore, $\tau^X$ is a stopping time for any stopping process $X$. Clearly,
$\tau^X$ is the first time of $X$ hitting the Lebesgue-measurable set $\{0\}$.

\begin{definition}
For a stopping time $\tau$, define a stochastic process $X^{\tau}=(X^{\tau}_t)_{t\in\mathbb{T}}$ by
\begin{equation}
\label{st22}
X^{\tau}_t = \mathbf{1}_{\{\tau > t\}} \quad (t\in\mathbb{T}) .
\end{equation}
\end{definition}

By $\{\tau > t\} = \{\tau \leq t\}^C$ and \eqref{st8}, $X^{\tau}$ is an adapted process.
One has $X^{\tau}_s = \mathbf{1}_{\{\tau > s\}} \geq \mathbf{1}_{\{\tau > t\}} = X^{\tau}_t$
for $s \leq t$. Because of 
$\lim_{t\downarrow\tau(\omega)} X^{\tau}_t(\omega) = 0 = X^{\tau}_{\tau(\omega)}$,
$X^{\tau}$ is c\`adl\`ag and hence a stopping process.

\begin{theorem}
\label{theorem1}
The mapping
\begin{eqnarray}
\label{f}
f: X & \longmapsto & \tau^X
\end{eqnarray}
is a bijection between the stopping processes and the stopping times 
on $(\Omega, \f_\infty, \mathbb{F}, \p)$ such that
\begin{eqnarray}
\label{f-1}
f^{-1}: \tau & \longmapsto & X^{\tau} ,
\end{eqnarray}
and hence
\begin{equation}
\label{tau_tau}
\tau^{X^{\tau}} = \tau \quad \text{and} \quad X^{\tau^X} = X .
\end{equation}
\end{theorem}

\begin{proof}
That $f$ maps stopping processes $X$ to stopping times $\tau^X$ was seen in \eqref{st7}.
That different stopping processes lead to different stopping times under $f$ is obvious 
from \eqref{st6}; $f$ is therefore an injection.
For any stopping time $\tau$, one has by \eqref{st6} and \eqref{st22} that 
\begin{equation}
\label{st9}
\tau^{X^{\tau}} (\omega)  = 
\begin{cases}
+\infty & \quad \text{if } \mathbf{1}_{\{\tau>t\}}(\omega) = 1 \; \text{ for all } t\in\mathbb{T} , \\
\min \{t\in\mathbb{T}: \mathbf{1}_{\{\tau>t\}}(\omega) = 0\} & \quad \text{otherwise} .
\end{cases}
\end{equation}
Since $\mathbf{1}_{\{\tau>t\}}(\omega) = 0$ if and only if $\tau(\omega) \leq t$, the right
hand side of \eqref{st9} is $\tau (\omega)$. Therefore, $\tau (\omega) = \tau^{X^{\tau}} (\omega)$,
and $f$ is a surjection and therefore a bijection. 
From \eqref{st2} and \eqref{st22}, we obtain $X_t^{\tau^X} = \mathbf{1}_{\{\tau^X > t\}} = X_t$
for $t\in\mathbb{T}$, which proves \eqref{f-1}.
\end{proof}

Note that there was no identification of almost surely identical stopping times or stopping processes
in Theorem \ref{theorem1} or any of the definitions, however, one can obviously transfer
all results to equivalence classes of almost surely identical objects.



\end{document}